\documentclass[11pt]{amsart}
\usepackage{amssymb,amsmath,amsthm}
\usepackage{amsfonts}
\usepackage{enumerate}
\usepackage{dsfont}
\usepackage{cite}
\usepackage[colorlinks, citecolor=red]{hyperref}
\usepackage{mathrsfs}
\usepackage{epsfig}
\usepackage{lscape}
\usepackage{subfigure}
\usepackage{epstopdf}
\usepackage{caption}
\usepackage{algorithm}
\usepackage{algpseudocode}
\usepackage{multirow}
\usepackage{geometry}
\usepackage{paralist}
\usepackage{enumerate}
\usepackage{url}
\usepackage{graphicx,cite}
\usepackage{longtable}
\usepackage{tikz}
\usepackage[normalem]{ulem}
\usepackage{framed}

\newcommand{\seqk}[1]{\{ #1 \}_{k\geq0}}
\newcommand{\hf}{\frac{1}{2}}

\newcommand{\norm}[1]{\|{#1}\|_2}

\newcommand{\E}{{\mathbb E}}

\newcommand{\R}{{\mathbb R}}
\newcommand{\Rn}{{\mathbb R}^n}

\renewcommand{\i}{{\mathbf i}}

\newtheorem{definition}{Definition}[section]

\newtheorem{theorem}[definition]{Theorem}
\newtheorem{lemma}[definition]{Lemma}

\newtheorem{remark}[definition]{Remark}

\allowdisplaybreaks

\renewcommand{\t}{^\top}
\renewcommand{\i}{_i}
\newcommand{\s}{^\star}
\renewcommand{\k}{^{(k)}}
\newcommand{\kp}{^{(k+1)}}
\newcommand{\km}{^{(k-1)}}

\newcommand{\zo}{^{(0)}}

\newcommand{\Ik}{\mathcal{I}_k}

\newcommand{\inn}[1]{\langle #1 \rangle}

\newcommand{\innL}[1]{\bigg\langle #1 \bigg\rangle}
\newcommand{\nm}[1]{\| #1 \|_2}
\newcommand{\nmsq}[1]{\| #1 \|_2^2}
\newcommand{\nmsqL}[1]{\left\| #1 \right\|^2}
\newcommand{\nmfsq}[1]{\| #1 \|^2_F}
\newcommand{\ik}{_{i_k}}
\newcommand{\ikm}{_{i_{k-1}}}
\newcommand{\ikmm}{_{i_{k-2}}}
\newcommand{\bk}{\beta_k}
\newcommand{\bkm}{\beta_{k-1}}
\newcommand{\prob}{\operatorname{Prob}}

\newcommand{\sima}{\sigma_{\max}}
\newcommand{\simi}{\sigma_{\min}}

\newcommand{\MAXTERM}{\max\limits_{i \in [m]} \left\{ \frac{ |\res{\i}{x\k}|^2}{\nmsq{a\i}} \right\}}
\newcommand{\err}[1]{\nmsq{#1-x\s}}
\newcommand{\res}[2]{\inn{a#1,#2}-b#1}
\newcommand{\fres}[2]{\frac{\inn{a#1,#2}-b#1}{\nmsq{a#1}}}
\date{}
\begin{document}

	\title[GMIRK for linear systems]{On greedy multi-step inertial randomized Kaczmarz method for solving linear systems}
	
	\author{Yansheng Su}
	\address{School of Mathematical Sciences, Beihang University, Beijing, 100191, China. }
	\email{suyansheng@buaa.edu.cn}
	
	\author{Deren Han}
	\address{LMIB of the Ministry of Education, School of Mathematical Sciences, Beihang University, Beijing, 100191, China. }
	\email{handr@buaa.edu.cn}

	\author{Yun Zeng}
	\address{School of Mathematical Sciences, Beihang University, Beijing, 100191, China. }
	\email{zengyun@buaa.edu.cn}
	
	\author{Jiaxin Xie}
	\address{LMIB of the Ministry of Education, School of Mathematical Sciences, Beihang University, Beijing, 100191, China. }
	\email{xiejx@buaa.edu.cn}
	
	\maketitle
	
	\begin{abstract}
		The multi-step inertial randomized Kaczmarz (MIRK) method is an iterative method for solving large-scale linear systems. In this paper, 	we enhance the MIRK method by incorporating the greedy probability criterion, coupled with the introduction of a tighter threshold parameter for this criterion. We prove that the proposed greedy MIRK (GMIRK) method enjoys an improved deterministic linear convergence  compared to both the MIRK method and the greedy randomized Kaczmarz method. 
			Furthermore, we exhibit that the multi-step inertial extrapolation approach can be geometrically interpreted as an orthogonal projection method, and  establish its relationship with the sketch-and-project method in (SIAM J. Matrix Anal. Appl. 36(4):1660-1690, 2015) and the oblique projection technique in (Results Appl. Math. 16:100342, 2022).
			Numerical experiments are provided to confirm our results.
	\end{abstract}
	\let\thefootnote\relax
	
	\let\thefootnote\relax
	
\section{Introduction}
We consider the following system of linear equations
\begin{equation}
	\label{main-prob}
	Ax=b,
\end{equation}
where $A\in\mathbb{R}^{m\times n}$ and $b\in\mathbb{R}^m$.  The problem for solving the linear system \eqref{main-prob} comes up in many
fields of scientific computing  and engineering, including computerized tomography \cite{hansen2021computed}, signal processing \cite{Byr04}, optimal control \cite{Pat17}, and machine learning \cite{Cha08}.
Throughout this paper, we assume that the linear system \eqref{main-prob} is consistent, i.e. there exists a $ x\s $ such that $ Ax\s = b $.

The Kaczmarz method \cite{kaczmarz1937angenaherte}, also known as \textit{algebraic reconstruction technique} (ART) \cite{herman1993algebraic,gordon1970algebraic}, is a classic and effective row-action method for solving the linear system \eqref{main-prob}. For any $i\in[m]:=\{1,\ldots,m\}$, let $ a\i $ denote the transpose of the $i$-th row of $ A $ and $b_i$  denote the $i$-th  entry of $b$. Starting from an initial point  $x^{(0)}\in\mathbb{R}^n$, the original Kaczmarz method constructs the next iterate $x\kp$ via the following update rule
\begin{equation}\label{equ:basiciter}
	x\kp = x\k - \frac{a\ik\t x\k - b\ik}{\nmsq{a\ik}} a\ik,
\end{equation}
where the index $ i_k = (k \mod m) +1$.
In some cases, the Kaczmarz method may converge slowly when the rows of the coefficient matrix  are in an unfavorable order. This phenomenon may occur, for example, when there are two adjacent rows that are very close to each other. To address this issue, one approach is to select a working row from the matrix $A$ randomly rather than sequentially \cite{herman1993algebraic,natterer2001mathematics,feichtinger1992new,sun2021worstcase}.
The celebrated result of Strohmer and Vershynin \cite{strohmer2009randomized} shows that if the index $ i_k$ is selected randomly with probability proportional to $\|a_{i_k}\|^2_2$,  the resulting \textit{randomized  Kaczmarz} (RK) method converges linearly in expectation. This result has sparked significant interest in RK-type methods due to their computational efficiency and scalability.
Recently, Li et al \cite{LI2022100342} have introduced and analyzed the Kaczmarz method with oblique projection, where the next iteration is chosen to be on the intersection of two hyperplanes. Theoretical analysis demonstrated that the convergence rate of the RK method with oblique projection is faster than the RK method.

It is widely acknowledged that  using a better probability criterion can result in a more favorable order of working rows, thereby accelerating the convergence of the RK method.  To enhance the convergence property of the RK method,
Bai and Wu \cite{bai2018greedy} first introduced the \textit{greedy probability criterion} and developed the
\textit{greedy randomized Kaczmarz} (GRK) method for solving the linear system \eqref{main-prob}.
By employing this greedy probability criterion, entries of small magnitude in the residual vector $Ax\k-b$ are excluded from selection, ensuring the progress in each iteration of GRK. This property also leads to a faster convergence rate of GRK compared to RK.
In recent years,  there  has been a large amount of work on the refinements and extensions of the GRK method, such as the capped adaptive  sampling rule \cite{gower2021adaptive,yuan2022adaptively}, the greedy augmented randomized Kaczmarz method for inconsistent linear systems \cite{bai2021greedy}, the greedy randomized coordinate descent method \cite{bai2019greedy}, the  momentum variant of the GRK method \cite{su2023linear}, and the capped nonlinear Kaczmarz method \cite{zhang2022greedy}. In addition, we note that there has also been some work on non-random Kaczmarz methods \cite{chen2022fast,niu2020greedy,zhang2019new} inspired by the GRK method. We refer to \cite{bai2023convergence} and\cite{bai2023randomized} for recent surveys  on the Kaczmarz method.

The inertial extrapolation,  initially introduced by Polyak \cite{polyak1964some},  has been widely used
to improve the performance of iterative algorithms. The main characteristic of inertial type methods is that  the next iteration is determined by utilizing the previous two iterates. Recently, the inertial extrapolation approach has been introduced to accelerate the  RK method \cite{he2023inertial}, resulting in the development of two variants of the inertial RK method: the alternated inertial RK (AIRK) method and the multi-step inertial RK (MIRK) method.
The authors showed that the MIRK method is superior to the AIRK method in both theoretical analysis and practical applications.
However, the geometric explanation of MIRK is rather vague while the AIRK method has a connection with the two-subspace Kaczmarz method \cite{needell2013two,wu2022two}.

In this paper, we incorporate the greedy probability criterion into the MIRK method and propose the greedy MIRK (GMIRK) method for solving linear systems.
Particularly, we propose a tighter threshold parameter for the greedy probability criterion.
Theoretical analysis shows that the GMIRK method achieves a faster convergence rate compared to  both the MIRK method and the GRK method.
Furthermore, linear convergence of the proposed GMIRK is based on
the quantity $\nmsq{x\k -x\s}$, rather than its expectation $ \E [\nmsq{x\k -x\s} ]$, which is commonly used in the literature 
\cite{bai2018greedy,bai2021greedy}. We refer such convergence as \textit{deterministic}.
Numerical results demonstrate that the GMIRK method outperforms the GRK method in terms of both iteration counts and computing times, especially when the rows are highly coherent.

In addition, we give an elaborate geometric explanation for the multi-step inertial extrapolation approach. Recall that the Kaczmarz method orthogonally projects the iterate $ x\k $ onto the solution set $ \{ x \ | \ \inn{a\ik,x} = b\ik\} $ at each iteration. In this paper, we exhibit that the Kaczmarz method with multi-step inertial extrapolation approach actually projects $ x\k $ onto the solution set of two successive working rows
\[
\left\{ x \ \middle\vert \
\begin{bmatrix} a\ikm^\top \\a\ik^\top \end{bmatrix} x = \begin{bmatrix} b\ikm \\b\ik \end{bmatrix}  \right\}.
\]
This geometric explanation provides further insight into why the multi-step inertial extrapolation approach is faster.
We further establish the connection between our algorithm and the sketch-and-project method \cite{15M1025487,19M1285846}, which is a general framework of randomized iterative methods for linear systems. We also exhibit that the randomized Kaczmarz method with oblique projection \cite{LI2022100342} can also be interpreted properly within these geometric perspectives.

\subsection{Notations and preliminaries}
We here give some notations that will be used in the paper. For vector $x\in\mathbb{R}^n$, we use $x_i,x\t$ and $\|x\|_2$ to denote the $i$-th entry, the transpose and the Euclidean norm of $x$, respectively. For matrix $A\in\mathbb{R}^{m\times n}$, we use $a_i$, $A\t$, $A^\dagger$, $ \|A\|_F $, $ \operatorname{Range}(A) $, $\text{Rank}(A)$, $ \sima(A) $, and $ \simi(A) $ to denote the $i$-th row, the transpose, the Moore-Penrose pseudoinverse, the Frobenius norm, the range space, the rank, the largest and the smallest non-zero singular value of $ A $, respectively. For any $i\in[m]$, let $P_i$ denote the orthogonal projection operator onto the hyperplane  $ H_i = \{x \mid \langle a_i,x\rangle=b_i\}$, i.e. for any $x\in \mathbb{R}^n$,
$$P_{i}(x)=x-\frac{\langle a_i,x\rangle-b_i}{\|a_i\|^2_2}a_i.$$
We define
\begin{equation}\label{def-gamma}
	\gamma_1 := \max\limits_{i \in [m]} \sum\limits_{\substack{j=1, \\ j\neq i}}^{m} \nmsq{a_j}, \quad \gamma_2 := \max\limits_{\substack{i,j \in [m] \\ i\neq j}} \sum\limits_{\substack{k=1, \\ k\neq i,j}}^{m} \nmsq{a_k},
\end{equation}
and
\begin{equation}\label{def-delta}
	\delta_{\min}:=\min_{ i,j\in [m],i\neq j}\left|\innL{\frac{a_i}{\nm{a_i}},\frac{a_j}{\nm{a_j}}}\right|.
\end{equation}
Throughout this paper, we assume that for any $i,j\in[m]$ and $i\neq j$, $a_i$ and $a_j$ are linearly independent. Hence, it is obvious that  $ \gamma_2 < \gamma_1 < \nmfsq{A}$ and  $0\leq\delta_{\min}<1$.

\subsection{Organization}
The remainder of the paper is organized as follows. In Sectoin 2, we introduce the GMIRK, analyze its convergence and present its geometric explanation. In Section 3, we perform some numerical experiments to show the effectiveness of the proposed method. Finally, we conclude this paper in Section 4.

\section{The greedy multi-step inertial randomized Kaczmarz method}
In this section, we present the \textit{greedy multi-step inertial  randomized Kaczmarz} (GMIRK) method for solving the linear system \eqref{main-prob}. The algorithm is formally described in Algorithm \ref{algo:main}, where the parameter sequence $\{\beta_k\}_{k\geq0}$ is carefully designed (see Remark \ref{remark-xie-2}).

\begin{remark}
		GMIRK is namely a combination of GRK and MIRK. We apply a modified variant of greedy probability criterion of GRK for selecting the rows and use the iteration scheme from MIRK to operate the iterates. Based on Algorithm \ref{algo:main}, if one selects $i_k$ from $[m] \setminus \{i_k\} $ with probability 
		\[
		\prob(i_k = i) = \frac{\nmsq{a_i}}{\nmfsq{A}-\nmsq{a\ikm}},
		\]
		one will recover MIRK; If one applies the standard Kaczmarz iteration \eqref{equ:basiciter} for obtaining $x\kp$, one will get GRK (with the $\Gamma_k$ in $\varepsilon_k$ being fixed as $\nmfsq{A}$). One may refer to \cite{bai2018greedy} and \cite{he2023inertial} for more details of GRK and MIRK respectively.
	\end{remark}

\begin{algorithm}[htpb]
	\caption{Greedy multi-step inertial randomized Kaczmarz (GMIRK) method \label{algo:main}}
	\begin{algorithmic}
		\rmfamily
		\Require
		$A\in \mathbb{R}^{m\times n}$, $b\in \mathbb{R}^m$, $ \theta \in [0,1] $, $k=0$ and initial point $x\zo \in\mathbb{R}^n$.
		\begin{enumerate}
			\item[1:] Compute
			\begin{equation}\label{equ:alo-eps}
				\varepsilon_k = \frac{1}{2}\left(\frac{1}{\nmsq{Ax\k-b}} \max_{i\in[m]} \left\{ \frac{\left|a\i\t x\k- b\i\right|^2}{\nmsq{a\i}} \right\} + \frac{1}{\Gamma_k}\right).
			\end{equation}
			Here $ \Gamma_0 = {\nmfsq{A}} $, $ \Gamma_1 = {\gamma_1} $ and for $k\geq2$, $ \Gamma_k = {\gamma_2}$, where $\gamma_1$ and $\gamma_2$ are defined as \eqref{def-gamma}.
			\item[2:] Determine the index set
			\begin{equation}\label{alg-xie-1}
				\Ik = \left\{  i \ \left| \ \left|a\i\t x\k- b\i\right|^2 \geq \varepsilon_k \nmsq{Ax\k-b}\nmsq{a\i} \right. \right\}.
			\end{equation}
			\item[3:] Select $ i_k \in  \Ik $ according to some probability
			$$
			\prob(i_k = i) = p\k_i,
			$$			
			where $ p\k_i = 0 \operatorname{if} i \notin \Ik $, $ p\k_i \geq 0 \operatorname{if} i \in \Ik $, and $ \sum_{i\in\Ik} p\k_i=1 $.
			\item[4:]Compute	
			\[
			\bk = \left\{
			\begin{array}{ll}
				0, & \hbox{if $k=0$;} \\
				\frac{\inn{a\ik,a\ikm}(\res{\ik}{x\k})}{\nmsq{a\ik}\nmsq{a\ikm} - \inn{a\ik,a\ikm}^2}, & \hbox{otherwise.}
			\end{array}
			\right.
			\]
			\item[5:] Set
			\begin{equation}\label{xie-eq-0711-1}
				w\k = \left\{
				\begin{array}{ll}
					x\k, & \hbox{if $k=0$;} \\
					x\k + \bk a\ikm, & \hbox{otherwise.}
				\end{array}
				\right.
			\end{equation}
			\item[6:] Set
			\[
			x\kp = w\k - \frac{\res{\ik}{w\k}}{\nmsq{a\ik}}a\ik.
			\]
			\item[7:] If the stopping rule is satisfied, stop and go to output. Otherwise, set $k=k+1$ and return to Step $1$.
		\end{enumerate}
		
		\Ensure
		The approximate solution $x\k$.
	\end{algorithmic}
\end{algorithm}


It can be observed that when $k=0$, Algorithm \ref{algo:main} and the GRK method utilize the same iteration. For $k\geq1$, since $\langle a_{i_k}, x\k\rangle-b_{i_k}\neq 0$ (otherwise, $Ax\k-b=0$ would indicate the convergence of GMIRK), it can be verified that $x\k\neq w^{(k-1)}$.
Consequently,
$$
	x^{(k+1)}=P_{i_k}(x\k+\beta'_k(x\k-w^{(k-1)})) = P_{i_k}(x\k+\beta'_k(x\k-x^{(k-1)})-\beta'_k\beta_{k-1}a_{i_{k-2}}),
	$$
where
$$
\beta'_k=-\frac{(\langle a_{i_k}, x\k\rangle-b_{i_k})\langle a_{i_{k-1}},a_{i_k}\rangle \|a_{i_{k-1}}\|^2_2}{(\langle a_{i_{k-1}}, w^{(k-1)}\rangle-b_{i_{k-1}})(\|a_{i_{k-1}}\|^2_2\|a_{i_{k}}\|^2_2-\langle a_{i_{k-1}},a_{i_k}\rangle^2 )}.
$$
The inertial term $\beta'_k(x\k-x\km)$ implies that Algorithm \ref{algo:main} is a multi-step inertial algorithm \cite{dong2019mikm}. 
We refer the reader to Remark 4.1 in \cite{he2023inertial} for more details.

Comparing to the original greedy probability criterion \cite{bai2018greedy}, there are mainly two modifications in the row selection rule of our method. After determining $ \Ik $, the original criterion requires to compute an extra vector $ \tilde{r}\k $ defined component-wise by
\begin{equation}\label{rk-xie-0712}
	\tilde{r}\k_i= \begin{cases}a_{i}^\top x\k-b_i, & \text { if } i \in \mathcal{I}_k, \\ 0, & \text { otherwise},\end{cases}
\end{equation}
and to select $ i_k $	with probability $
\operatorname{Prob}\left(i_k=i\right)=\frac{\left|\tilde{r}\k_{i}\right|^2}{\|\tilde{r}\k\|_2^2}.$ In our method, it is generalized into any valid probability $ \prob(i_k = i) = p\k_i $,  as long as $ p\k_i = 0  $ when $ i \notin \Ik $. This generalization allow the users to adopt more flexible probability criteria. For instance, one may reduce the computational cost of each iteration by using simpler probabilities, such as selecting $i_k$ uniformly from $\Ik$.

Another modification arises in the threshold parameter $ \varepsilon_k$. Compared to the $ \varepsilon_k $ employed in \cite{bai2018greedy}, i.e.
\[
\varepsilon_k = \frac{1}{2}\left(\frac{1}{\nmsq{Ax\k-b}} \max_{i\in[m]} \left\{ \frac{\left|a\i\t x\k- b\i\right|^2}{\nmsq{a\i}} \right\} + \frac{1}{\nmfsq{A}}\right),
\]
the last summand in \eqref{equ:alo-eps} is changed into $ \frac{1}{\Gamma_k} $. This modification serves to provide a tighter threshold parameter $\varepsilon_k$, as $ \frac{1}{\Gamma_k}> \frac{1}{\nmfsq{A}}$ for $k\geq 1$. 

Next, we show that the GMIRK method is well defined as index set $\mathcal{I}_k$ defined in \eqref{alg-xie-1} is nonempty for $k\geq0$. Let us first establish two useful lemmas.
The first one describes that the residuals of $ x\k $ on $ H\ikm $ and $ H\ikmm $ are zero (when $ k\geq2 $).
\begin{lemma}
	Suppose that $ \seqk{x\k} $ is the sequence generated by Algorithm \ref{algo:main}.  Then for any $ k\geq 1$,
	\begin{equation}\label{xie-proof-2}
		r\k_{i_{k-1}} = \res{\ikm}{x\k} = 0, 			
	\end{equation}
	and for any $ k\geq 2$,
	\begin{equation}\label{xie-proof-2p}
		r\k_{i_{k-2}} = \res{\ikmm}{x\k} = 0.
	\end{equation}
\end{lemma}
\begin{proof}
	For $ k \geq 1 $, we have
	\begin{align*}
		r\k_{i_{k-1}} = &  \res{\ikm}{x\k}\\
		=  & \innL{a\ikm, w\km - \fres{\ikm}{w\km} a\ikm } - b\ikm \\
		= &  \res{\ikm}{w\km} - \big(\res{\ikm}{w\km}\big)\\
		=&  0.
	\end{align*}
	For $ k \geq 2 $, we have
	\begin{align*}
		r\k\ikmm = &  \res{\ikmm}{x\k} \\
		=  &   \innL{a\ikmm, w\km - \fres{\ikm}{w\km} a\ikm } - b\ikmm \\
		= &   \innL{a\ikmm, (x\km + \bkm a\ikmm) - \fres{\ikm}{x\km + \bkm a\ikmm} a\ikm } - b\ikmm \\		
		= &  \res{\ikmm}{x\km} + \bkm\frac{\nmsq{a\ikm}\nmsq{a\ikmm} - \inn{a\ikm,a\ikmm}^2}{\nmsq{a\ikm}} 
		\\
		& - \frac{\inn{a\ikm,a\ikmm}(\res{\ikm}{x\km})}{\nmsq{a\ikm}} \\
		= &  \res{\ikmm}{x\km}\\
		=& 0,
	\end{align*}
	where the last equality follows from \eqref{xie-proof-2}.	
\end{proof}
\begin{lemma}\label{lemma2}
	Suppose that $ \Gamma_0 = {\nmfsq{A}} $, $ \Gamma_1 = {\gamma_1} $ and for $k\geq2$, $ \Gamma_k = {\gamma_2}$, where $\gamma_1$ and $\gamma_2$ are defined as \eqref{def-gamma}.
	Let $ \seqk{x\k} $ be the sequence generated by Algorithm \ref{algo:main}. Then for any $ k \geq 0 $,
	\[
	\max_{i\in[m]} \left\{  \frac{\left|\res{\i}{x\k}\right|^2}{\nmsq{a\i}} \right\}	 \geq \frac{1}{\Gamma_k}\nmsq{Ax\k-b}.
	\]
\end{lemma}
\begin{proof}
	When $ k = 0, $
	\[
	\max_{i\in[m]} \left\{  \frac{\left|\res{\i}{x\k}\right|^2}{\nmsq{a\i}} \right\}	 \geq \sum_{i=1}^m \frac{\nmsq{a\i}}{\nmfsq{A}} \frac{\left|\res{\i}{x\k}\right|^2}{\nmsq{a\i}} = \frac{1}{\nmfsq{A }}\nmsq{Ax\k-b}.
	\]
	This is because the largest term in $ \bigg\{ \frac{\left|\res{\i}{x\k}\right|^2}{\nmsq{a\i}} \bigg\}_{i=1}^m$ is not less than any convex combination of all the terms. Similarly, when $ k = 1, $
	\begin{align*}\label{xie-proof-0729-1}
		\max_{i\in[m]} \left\{  \frac{\left|\res{\i}{x\k}\right|^2}{\nmsq{a\i}} \right\}	 \geq & \sum\limits_{\substack{i=1, \\ i\neq i_{k-1} }}^{m} 	\frac{\nmsq{a\i}}{\nmfsq{A}-\nmsq{a\ikm }} \frac{\left|\res{\i}{x\k}\right|^2}{\nmsq{a\i}} \\ \geq & \  \frac{1}{\gamma_1} \sum\limits_{\substack{i=1, \\ i\neq i_{k-1} }}^{m} \left|\res{\i}{x\k}\right|^2
		\\= & \ \frac{1}{\gamma_1} \sum\limits_{i=1}^{m} \left|\res{\i}{x\k}\right|^2  \\
		= & \ \frac{1}{\gamma_1} \nmsq{Ax\k-b},
	\end{align*}
	where the first equality follows from \eqref{xie-proof-2}. When $ k\geq2, $
	\begin{align*}
		\max_{i\in[m]} \left\{  \frac{\left|\res{\i}{x\k}\right|^2}{\nmsq{a\i}} \right\}	 \geq & 	\sum\limits_{\substack{i=1,  \\ i\neq i_{k-1}, i_{k-2} }}^{m} 	\frac{\nmsq{a\i}}{\nmfsq{A}-\nmsq{a\ikm }-\nmsq{a\ikmm}} \frac{\left|\res{\i}{x\k}\right|^2}{\nmsq{a\i}} \\
		\geq & \ \frac{1}{\gamma_2} \sum\limits_{\substack{i=1, \\ i\neq i_{k-1},i_{k-2} }}^{m} 	\left|\res{\i}{x\k}\right|^2
		\\=& \ \frac{1}{\gamma_2} \sum\limits_{i=1}^{m} \left|\res{\i}{x\k}\right|^2  \\
		= & \ \frac{1}{\gamma_2} \nmsq{Ax\k-b} ,
	\end{align*}
	where the first equality follows from \eqref{xie-proof-2} and \eqref{xie-proof-2p}.
\end{proof}
By the definition of $ \Ik $ in \eqref{alg-xie-1}, we know that $ i \in \Ik $ if and only if
\[
\frac{|\res{\i}{x\k}|}{\nmsq{a\i}} \geq \hf \left(  \max_{i\in[m]} \left\{  \frac{\left|\res{\i}{x\k}\right|^2}{\nmsq{a\i}} \right\} + \frac{1}{\Gamma_k}\nmsq{Ax\k-b}  \right).
\]
Lemma \ref{lemma2} shows that at least $ i_k^{\max}$ belongs to $ \Ik $, where $ i_k^{\max} \in \arg\max\limits_{i\in[m]} \bigg\{  \frac{\left|\res{\i}{x\k}\right|^2}{\nmsq{a\i}} \bigg\} $. Thus $ \Ik $ is nonempty and  the GMIRK method is well defined.

\subsection{Convergence analysis}

We have the following convergence result for Algorithm \ref{algo:main}.

\begin{theorem}\label{theo:basic}
	Suppose that $ x\zo \in \Rn $ and let $ x\s = A^\dagger b + (I-A^\dagger A)x\zo $ denote the projection of $ x\zo $ onto the solution set of $ Ax=b$. Then the iteration sequence $ \seqk{x\k} $ generated by Algorithm \ref{algo:main} satisfies
	$$
	\nmsq{x^{(1)}-x\s} \leq  \rho_0 \nmsq{x\zo-x\s}
	$$
	and
	$$
	\nmsq{x\k-x\s} \leq \rho_2^{k-2}\rho_1 \rho_0 \nmsq{x\zo-x\s}, \  k=2,3,\ldots,
	$$
	where
	\[
	\rho_0 =  1-\frac{\simi^2(A)}{\nmfsq{A}},\ \ \rho_1= 1- \frac{\simi^2(A)}{(1-\delta_{\min}^2)\gamma_1} , \ \ \text{and} \ \ \rho_2 =  1- \frac{\simi^2(A)}{(1-\delta_{\min}^2)\gamma_2}  ,
	\]
	and $ \gamma_1$, $\gamma_2$, and $\delta_{\min}$ are defined as \eqref{def-gamma} and \eqref{def-delta}, respectively.
\end{theorem}
\begin{proof}
	By the iterative strategy of Algorithm \ref{algo:main}, for $k\geq1$, we have
	\begin{align}
		\err{x\kp}= & \ \nmsqL{w\k -x\s -\fres{\ik}{w\k}a\ik} \nonumber \\
		= & \ \err{w\k} - \frac{\left|\res{\ik}{w\k}\right|^2}{\nmsq{a\ik}}  \nonumber \\
		= & \ \nmsq{x\k-x\s + \bk a\ikm} -  \frac{\left|\res{\ik}{x\k} + \bk\inn{a\ik,a\ikm}\right|^2 }{\nmsq{a\ik}} \nonumber \\
		\overset{\textcircled{a}}{=}& \ \err{x\k} + \bk^2 \nmsq{a\ikm} - \frac{\left|\res{\ik}{x\k}\right|^2}{\nmsq{a\ik}} - \bk^2 \frac{\inn{a\ik,a\ikm}^2}{\nmsq{a\ik}} 
		\nonumber \\ & -  2\bk \frac{\inn{a\ik,a\ikm}(\res{\ik}{x\k})}{\nmsq{a\ik}} \nonumber \\
		= & \ \err{x\k} - \frac{\left|\res{\ik}{x\k}\right|^2}{\nmsq{a\ik}} 
		+ \bk^2 \ \frac{\nmsq{a\ik}\nmsq{a\ikm} - \inn{a\ik,a\ikm}^2}{\nmsq{a\ik}}
		\nonumber \\ &
		- 2\bk \frac{ \inn{a\ik,a\ikm}(\res{\ik}{x\k}) }{\nmsq{a\ik}}, \label{xie-proof-1} 
	\end{align}
	where $\textcircled{a}$ follows from the fact that $\langle a\ikm,x\k-x\s\rangle= \res{\ikm}{x\k} =r\k_{i_{k-1}}=0$.
	Then by the choices of $\beta_k$, for $ k\geq 1 $ we have
	\begin{align}
		\err{x\kp}= & \ \err{x\k} - \frac{\left|\res{\ik}{x\k}\right|^2}{\nmsq{a\ik}}  \nonumber \\ & 
		- \frac{\inn{a\ik,a\ikm}^2}{\nmsq{a\ik}\nmsq{a\ikm} - \inn{a\ik,a\ikm}^2} \frac{\left|\res{\ik}{x\k}\right|^2}{\nmsq{a\ik}} \nonumber \\
		= & \ \err{x\k} - \frac{\nmsq{a\ik}\nmsq{a\ikm}}{\nmsq{a\ik}\nmsq{a\ikm} - \inn{a\ik,a\ikm}^2} \frac{\left|\res{\ik}{x\k}\right|^2}{\nmsq{a\ik}} \nonumber \\
		= & \ \err{x\k} - \frac{1}{1 - \left\langle\frac{a\ik}{\nm{a\ik}},\frac{a\ikm}{\nm{a\ikm}}\right\rangle^2} \frac{\left|\res{\ik}{x\k}\right|^2}{\nmsq{a\ik}} \nonumber \\				
		\overset{\textcircled{b}}{\leq} & \ \err{x\k} - \frac{1}{1 - \delta_{\min}^2} \frac{\left|\res{\ik}{x\k}\right|^2}{\nmsq{a\ik}} \nonumber \\
		\overset{\textcircled{c}}{\leq} &  \nmsq{x\k-x\s} - \frac{\varepsilon_k}{1 - \delta_{\min}^2} \nmsq{Ax\k-b}, \label{xie-proof-3}
	\end{align}
	where $\textcircled{b}$ follows from the definition of the $\delta_{\min}$ and $\textcircled{c}$ follows from \eqref{alg-xie-1}.
	For the case $k=0$, by using the similar argument we can get
	\begin{equation}\label{xie-proof-4}
		\|x^{(1)}-x\s\|^2_2\leq\|x^{(0)}-x\s\|^2_2- \varepsilon_0\nmsq{Ax^{(0)}-b}.
	\end{equation}
	Next, let us give an estimate for the quantity $\varepsilon_k$. In fact, we have
	\[
	\varepsilon_k =  \frac{1}{2\|Ax\k-b\|^2_2}  \MAXTERM+\frac{1}{2\Gamma_k}
	\geq  \frac{1}{2\Gamma_k} + \frac{1}{2\Gamma_k}  = \frac{1}{\Gamma_k},
	\]
	where the inequality follows from Lemma \ref{lemma2}. Hence, by the definition of $ \Gamma_k $, we have
	\begin{equation}\label{equ:basicprof1}
		\varepsilon_k \geq \left\{ \begin{array}{ll}
			\frac{1}{\nmfsq{A}}, & k = 0, \\
			\dfrac{1}{\gamma_1},  & k=1, \\
			\dfrac{1}{\gamma_2},  & k\geq2.
		\end{array}\right.
	\end{equation}	
	Now, let us give an estimate for $\nmsq{Ax\k -b}$. We have that for any $k\geq0$, $ x\k - x\s \in \text{Range}(A^\top) $. Indeed, from the definition of $ x\s $, we know that $ x\zo - x\s = A^\dagger(Ax\zo - b) \in \text{Range}(A^\top) $. Suppose that $ x\k - x\s \in \text{Range}(A^\top) $ holds, then $w^{(k)}- x\s=x\k- x\s+\beta_k a_{i_{k-1}}\in \text{Range}(A^\top)$. Hence
	$ x\kp - x\s = w\k -x\s -  \frac{a\ik\t w\k - b\ik}{\nmsq{a\ik}}a\ik \in \text{Range}(A^\top) $. By induction we have that  $ x\k - x\s \in \text{Range}(A^\top) $ holds for any $ k\geq 0. $ Therefore,
	$$
	\nmsq{Ax\k -b} = \nmsq{A(x\k-x\s)} \geq \simi^2(A) \nmsq{x\k-x\s}.
	$$
	Substituting this and \eqref{equ:basicprof1} into \eqref{xie-proof-3} and \eqref{xie-proof-4} completes the proof.
\end{proof}

\begin{remark}\label{remark-xie-2}
	It can be observed from \eqref{xie-proof-1} that the parameter $\beta_k$ in Algorithm \ref{algo:main} is carefully chosen such that  $\|x\kp-x\s\|^2_2$ is minimized.
\end{remark}

\begin{remark}\label{xie-remark-1}
	We here compare our convergence result with those of GRK and MIRK respectively. In Theorem 2.1 of \cite{su2023linear}, the authors showed that the  iteration sequence $\{x\k\}_{k\geq0}$ of the GRK method satisfies
	$$		\nmsq{x\k-x\s} \leq \bigg(  1-\hf\bigg( \frac{\nmfsq{A}}{\gamma_1}+1 \bigg) \frac{\simi^2(A)}{\nmfsq{A}} \bigg)^{k-1} \bigg(1-\frac{\simi^2(A)}{\nmfsq{A}}\bigg) \nmsq{x\zo-x\s}, \ k\geq1.
	$$
	It is obvious that $\rho_2<\rho_1 <\left(  1-\hf\left( \frac{\nmfsq{A}}{\gamma_1}+1 \right) \frac{\simi^2(A)}{\nmfsq{A}} \right)$, which means that the convergence factor of Algorithm \ref{algo:main} in Theorem \ref{theo:basic} is smaller than that of the GRK method.
	In Theorem 4.1 of \cite{he2023inertial}, the authors showed that the  iteration sequence $\{x\k\}_{k\geq0}$ of the MIRK method satisfies
	$$		
	\E [ \nmsq{x\k-x\s} ] \leq \rho_1^{k-1} \rho_0 \nmsq{x\zo-x\s}, \ k\geq1.
	$$
	Since $ \rho_2 < \rho_1$, we have that Algorithm \ref{algo:main} achieves a better convergence factor compared to the MIRK method. Furthermore, 
	our convergence result is for the quantity of $\nmsq{x\k -x\s}$ and we refer to such convergence for randomized algorithms as deterministic.
\end{remark}

\subsection{A geometric interpretation} In this subsection, we give several views of geometric interpretation of the multi-step inertial  extrapolation approach.
For any $i\in[m]$, we define the hyperplane
$$ {H\i = \{ x \mid \inn{a\i,x} = b\i\},} $$
and let
$$\Pi_k=x\k+\text{span}\{a_{i_k},a_{i_{k-1}}\}.$$

The procedure of the multi-step inertial  extrapolation approach consists of two steps.
Starting from $ x\k $, it first determines a specific point $ w\k $ along the direction $ a\ikm $, i.e. $ w\k = x\k +\bk a\ikm $. Then, it projects $ w\k $ onto the hyperplane $ H\ik $ to obtain $ x\kp $.
From \eqref{xie-proof-2} and \eqref{xie-proof-2p}, we know that the coefficient $ \bk $ is designed such that $ x\kp $ locates on the intersection of the two hyperplanes $ H\ik $ and $ H\ikm $. We present the geometric
interpretation of the multi-step inertial  extrapolation approach  in Figure \ref{fig:ink}.
\begin{figure}
	\centering
	\begin{tikzpicture}
		\draw (-1,0) -- (4.5,0);
		\draw (-0.3,-0.4) -- (2,2.667);
		\draw[dotted] (-0.3,3.350) -- (2.5,1.25);
		\draw[dashed, -stealth] (0,3.125) -- (0,0+0.05);
		\filldraw (1.5,2) circle [radius=1pt];
		\filldraw (0,3.125) circle [radius=1pt];
		\filldraw (0,0) circle [radius=1pt];
		
		\draw (1.5+0.5,2+0.15) node {$ x\k $};
		\draw (0-0.4,3.125-0.2) node {$ w\k $};
		\draw (0+0.5,0-0.3) node {$ x\kp $};
		\draw (4.5-0.8,0+0.3) node {$ H\ik = \{ x \mid \inn{a\ik,x} = b\ik \}  $};
		\draw (2+1.3,2.667+0.3) node {$ H\ikm =\{ x \mid \inn{a\ikm,x} = b\ikm \} $};
	\end{tikzpicture}
	\caption{ A geometric interpretation of the multi-step inertial extrapolation approach. The next iterate $x\kp$ is  the projection of $w\k=x\k+\beta_k a_{i_{k-1}}$ onto the hyperplane $H_{i_k}$ with $ \bk $ being chosen such that $ x\kp $ locates on the intersection of the two hyperplanes $ H\ik $ and $ H\ikm $.}
	\label{fig:ink}
\end{figure}
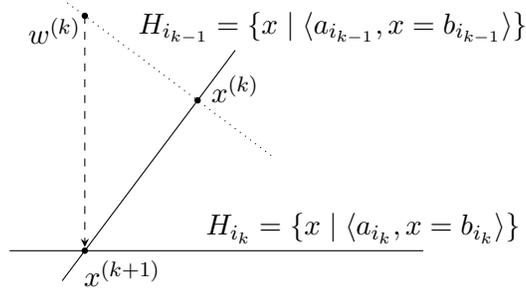
Meanwhile, we know that
$$
x\kp=w\k-\frac{\langle a_{i_k},w^k\rangle-b_{i_k}}{\|a_{i_k}\|^2_2}a_{i_k}=x\k+\beta_k a_{i_{k-1}}-\frac{\langle a_{i_k},w^k\rangle-b_{i_k}}{\|a_{i_k}\|^2_2}a_{i_k}\in \Pi_k.
$$
Therefore, $ x\kp \in H_{i_k} \cap H_{i_{k-1}}\cap \Pi_k$. Let
$$V_1 = \operatorname{span} \{ a\ik, a\ikm \} \  \text{and} \  V_2 = \left\{ x \ \left| \
\begin{bmatrix} a\ikm^\top \\a\ik^\top \end{bmatrix} x = 0 \right.\right\}. $$
It is apparent that those two subspaces are orthogonal and hence $ \{0\}=V_1 \cap V_2$. Since $V_1$ is parallel to $\Pi_k$ and $V_2$ is parallel to $H_{i_k} \cap H_{i_{k-1}}$ which implies that the intersection $H_{i_k} \cap H_{i_{k-1}}\cap \Pi_k$ is a singleton. Overall, we obtain that
\begin{equation}\label{xie-optimization0}
	\boxed{ \{x\kp\} = (H\ik \cap H\ikm) \ \bigcap \ \Pi_k.}
\end{equation}

This geometric interpretation can shed light on other perspectives. First, since $ x\kp $ is on $ \Pi_k $ and $ x\s - x\kp\in H\ik \cap H\ikm$ is perpendicular to $ \Pi_k $, i.e.
\[
\inn{x\s-x\kp,y-x\kp}=0, \quad \forall \ y \in \Pi_k,
\]
then $ x\kp $ is the projection of $ x\s $ onto $ \Pi_k, $ i.e.
\begin{equation}\label{xie-optimization}
	\boxed{x\kp =\arg\min  \|x-x\s\|^2_2 \ \ \text{subject to} \ \ x = x\k + \tau_1a\ik+\tau_2 a\ikm, \tau_1,\tau_2\in\mathbb{R}.}
\end{equation}
Simultaneously, $ x\kp $ is on $ H\ik \cap H\ikm $ and  $ x\k - x\kp \in V_1$ is perpendicular to $ H\ik \cap H\ikm $,  i.e.
\[
\inn{x\k-x\kp,y-x\kp}=0, \quad \forall \ y \in H\ik \cap H\ikm.
\]
Thus $ x\kp $ is also the projection of $ x\k $ onto $ H\ik \cap H\ikm$, i.e.
\begin{equation}\label{xie-optimization2}
	\boxed{x\kp =\arg\min  \|x-x\k\|^2_2 \ \ \text{subject to} \ \ \inn{a\ikm,x}=b\ikm, \ \inn{a\ik,x} = b\ik.}
\end{equation}
These three viewpoints are illustrated in Figure \ref{fig:soo}. Recall that the sketch-and-project \cite{15M1025487,19M1285846} update solves
\begin{equation*}
	\boxed{x\kp =\arg\min  \|x-x\k\|^2_B \ \ \text{subject to} \ \ S\ik\t Ax = S\ik\t b.}
\end{equation*}
Here $ \|  \cdot  \|_B = \sqrt{\inn{\cdot, B\cdot}}$ with $ B $ being symmetric positive defineite, and $ S\ik \in \R^{m\times \tau} $ is the \textit{sketching matrix} with sketching size $ \tau $. It is obvious that the multi-step inertial extrapolation approach provides a implementation for the above framework with $ B$ being an identity matrix and $ S_{i_k} = [e\ik,e\ikm], $ where $ e_i $ is the unit coordinate vector at the $ i $-th axis. Indeed, \eqref{xie-optimization0}, \eqref{xie-optimization}, and  \eqref{xie-optimization2} corresponse to the \textit{random-intersect} viewpoint, the \textit{constrain-and-approximate} viewpoint, and the \textit{sketch-and-project} viewpoint discussed in Section 2 of \cite{15M1025487} for the sketch-and-project method, respectively.

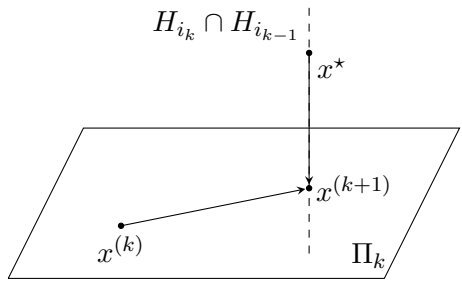
\begin{figure}[hptb]
	\centering
	\begin{tikzpicture}
		\draw (0,0)--(5,0)--(6,2)--(1,2)--(0,0);
		
		\filldraw (1.5,0.7) circle [radius=1pt]
		(4,3.0) circle [radius=1pt]
		(4,1.2) circle [radius=1pt];
		\draw (1.5,0.4) node {$x\k$};
		\draw (4.3,2.8) node {$x\s$};
		\draw (4.6,1.2) node {$x\kp$};
		\draw (4.8,0.3) node {$\Pi_k$};
		\draw (2.9,3.35) node {$H\ik \cap H\ikm$};
		\draw [dashed] (4,3.6) -- (4,0.25);
		\draw [-stealth] (4,3.0) -- (4,1.25);
		\draw [-stealth] (1.5,0.7) -- (3.95,1.19);
	\end{tikzpicture}
	\caption{The next iterate $x\kp$ is the only point in the intersection of $H\ik \cap H\ikm$ and $ \Pi_k$. It is the projection of $ x\s $ onto $ \Pi_k $ and meanwhile the projection of $ x\k $ onto $H\ik \cap H\ikm$.}	
	\label{fig:soo}
\end{figure}

	\begin{remark} \label{xie-remark-3}
		We now turn to the major advantages that projecting directly onto $H\ik \cap H\ikm$ brings about. It is well-known that the Kaczmarz iteration may suffer from slow convergence speed when the coefficient matrix $A$ is highly \textit{coherent}. A visualization of this is displayed in Figure \ref{fig:soo2}: For two coherent hyperplanes $H\ik$ and $H\ikm$, the iterates for a typical Kaczmarz method approach $x^*$ slowly, while MIRK and GMIRK are able to arrive at $x^*$ with the cost of around two Kaczmarz steps. The coherence of two rows can be expressed by the angle between their normal vectors $a\ik$ and $a\ikm$
		\[
		\cos \theta_{(a\ik,a\ikm)} = \left| \innL{\frac{a\ik}{\norm{a\ik}},\frac{a\ikm}{\norm{a\ikm}}} \right|,
		\] 
		and thus the coherence of the matrix $A$ can be bounded by 
		\[
		\delta_{\min}=\min_{ i,j\in [m],i\neq j}\left|\innL{\frac{a_i}{\nm{a_i}},\frac{a_j}{\nm{a_j}}}\right|.
		\]
		It is apparent from the definition of the coefficient $\rho_1$ and $\rho_2$ that when $\delta_{\min}$ is large, i.e., the coherence of $A$ is high, MIRK and GMIRK exhibit an accelerated convergence rate compared to the typical Kaczmarz iteration. A special case is $\delta_{\min} = 0$, where every $a_i(i\in[m])$ is orthogonal to each other. Then successive projections onto $H\i$ and $H_j$ will reach a certain point in the intersection of $H\i$ and $H_j$, which means the typical Kaczmarz iteration is also able to reach the intersection within the cost of two steps, and therefore MIRK and GMIRK show no advantage under this circumstance.
	\end{remark}

	\begin{remark} \label{xie-remark-4}
		There are also other Kaczmarz-type methods that share the same idea of ``projecting directly onto the intersection of two hyperplanes'', and they also share the same analytical frame in the acceleration for highly coherent matrices. The two-subspace Kaczmarz method \cite{needell2013two} is a renowned one of these methods. At each iteration, the method projects the iterate onto $H\ik$, and then projects the point onto an intermediate hyperplane that is orthogonal to $H\ik$. The intermediate hyperplane is generated by $H\ik$ and another chosen hyperplane $H_{j_k}$. In fact, the authors have established the equivalence of the two-subspace Kaczmarz method and AIRK in \cite{he2023inertial}, and highlighted their difference from MIRK and GMIRK. That is, although all the methods projects the current iterate onto the intersection of two hyperplanes at each iteration, the two-subspace Kaczmarz method and AIRK select the hyperplanes in individual pairs: $H_1$ and $H_2$, then $H_3$ and $H_4$, and so on, whereas MIRK and GMIRK select them in an overlapping manner: $H_1$ and $H_2$, then $H_2$ and $H_3$, and so forth.	
	\end{remark}

\begin{figure}
	\centering
	\begin{tikzpicture}
		
		\draw[] (0,0) -- (4.8, 1.2) node[anchor=west] {$H_{i_{k-1}}$};
		\draw[] (0,2) -- (4.8, 0.8) node[anchor=west] {$H_{i_{k}}$};
		
		\node at (4,1) [circle,fill,inner sep=1pt] {};
		\node at (4,1) [above] {$x^*$};
		
		\node at (0.8,1.2) [circle,fill,inner sep=1pt] {};
		\node at (0.8,1.2) [above left] {$x^{(k-1)}$};
		
		\node at (1.0353,0.2588) [circle,fill,inner sep=1pt] {};
		
		\node at (1.3841,1.6540) [circle,fill,inner sep=1pt] {};
		
		\node at (1.6918,0.4230) [circle,fill,inner sep=1pt] {};
		
		\node at (1.9634,1.5092) [circle,fill,inner sep=1pt] {};
		
		\node at (2.2030,0.5507) [circle,fill,inner sep=1pt] {};
		
		\node at (2.4144,1.3964) [circle,fill,inner sep=1pt] {};
		
		\node at (2.6010,0.6502) [circle,fill,inner sep=1pt] {};
		
		
		\draw[-stealth] (0.8,1.2) -- (1.0353-0.0107,0.2588+0.0427);
		
		\draw[-stealth] (1.0353,0.2588) -- (1.3841-0.0107,1.6540-0.0427);
		
		\draw[-stealth] (1.3841,1.6540) -- (1.6918-0.0107,0.4230+0.0427);
		
		\draw[-stealth] (1.6918,0.4230) -- (1.9634-0.0107,1.5092-0.0427);
		
		\draw[-stealth] (1.9634,1.5092) -- (2.2030-0.0107,0.5507+0.0427);
		
		\draw[-stealth] (2.2030,0.5507)  -- (2.4144-0.0107,1.3964-0.0427);
		
		\draw[-stealth] (2.4144,1.3964) -- (2.6010-0.0107,0.6502+0.0427);
		
		\draw[dashed, -stealth] (0.8,1.2) -- (4-0.05,1);
	\end{tikzpicture}
	\caption{A visualization of a typical Kaczmarz iteration that projects the iterates back-and-forth between the two hyperplanes $H\ik$ and $H\ikm$.}	
	\label{fig:soo2}
\end{figure}

We here also introduce another method that can establish connections with MIRK, the randomized Kaczmarz method with oblique projection\cite{LI2022100342,wang2022greedy}.
When projecting $ x\k $ onto the hyperplane $ H\ik $, the oblique projection does not use the orthogonal projection direction $ a\ik $. Instead, it uses a perturbed direction
$$ d\k = a\ik - \frac{\inn{a\ik,a\ikm}}{\nmsq{a\ikm}} a\ikm$$
derived from $ a\ik $ and it is orthogonal to $ a\ikm $.
Then setting $ x\kp = x\k+\eta_k d\k $ with $ \eta_k $ being designed to ensure that $ x\kp $ is on $ H\ik $. By the definition of $ d\k $, we know that $ x\kp $ is also on the $ H\ikm $. Besides, the next iterate $x\kp$ satisfies
$$
x\kp=x\k+\eta_k d\k=x\k+\eta_k a\ik - \eta_k\frac{\inn{a\ik,a\ikm}}{\nmsq{a\ikm}} a\ikm\in \Pi_k.
$$
Thus $ x\kp \in  H_{i_k} \cap H_{i_{k-1}}\cap \Pi_k $, which implies it is the same point that the multi-step inertial approach obtains.
The procedure is illustrated in Figure \ref{fig:obp}.

We note that Theorem 3.1 of \cite{wang2022greedy}, the authors had already investigated the greedy randomized Kaczmarz method with oblique projection. However, compared to Algorithm 2 in \cite{wang2022greedy}, we replace the term $\nmfsq{A}$ in the expression of $\varepsilon_k$  with $\Gamma_k$, which leads to a tighter convergence bound in Theorem \ref{theo:basic} than that in Theorem 3.1 of \cite{wang2022greedy}.
Besides, the convergence result in Theorem 3.1  of \cite{wang2022greedy} is  for the expectation $ \E [\nmsq{x\k -x\s} ]$, while the linear convergence in Theorem \ref{theo:basic} is deterministic. 

\begin{figure}
	\centering
	\begin{tikzpicture}
		\draw (-1,0) -- (4.5,0);
		\draw (-0.3,-0.4) -- (2,2.667);
		\draw[dashed] (1.5,0.5) -- (1.5,0);
		\draw[ -stealth] (1.5,2) -- (1.5,0.5+0.025);
		\draw[-stealth]  (1.5,2) -- (0.78,1.04);
		\draw[-stealth]  (0.78,1.04) -- (1.5-0.02,0.5+0.015);
		\filldraw (1.5,2) circle [radius=1pt];
		\filldraw (0,0) circle [radius=1pt];
		
		\draw (1.5+0.3,1.25) node {$ a\ik $};
		\draw (1.14-0.15,0.75-0.18) node {$ a\ikm $};
		\draw (1.14-2.1,1.5+0.2) node {$ a\ik - \frac{ \inn{a\ik,a\ikm}}{\nmsq{a\ikm}} a\ikm = d\k $};
		\draw (1.5+0.45,2) node {$ x\k $};
		\draw (0+0.5,0-0.3) node {$ x\kp $};
		\draw (4.5-0.8,0+0.3) node {$ H\ik = \{ x \mid \inn{a\ik,x} = b\ik \}  $};
		\draw (2+1.3,2.667+0.3) node {$ H\ikm =\{ x \mid \inn{a\ikm,x} = b\ikm \} $};
	\end{tikzpicture}
	\caption{A geometric explanation of the oblique projection technique. The next iterate $x\kp$ is  the projection of $x\k$ along the direction $d\k$ onto the hyperplane $H_{i_k}$.}
	\label{fig:obp}
\end{figure}
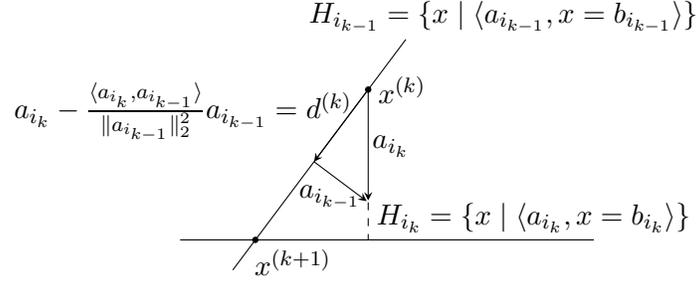

	\section{Numerical Experiments}

In this section, we describe some numerical results for the GMIRK method for solving linear systems. We also compare the GMIRK  method with the GRK method and the MIRK method on a variety of test problems.
All methods are implemented in  {\sc Matlab} R2022a for Windows $10$ on a desktop PC with the  Intel(R) Core(TM) i7-10710U CPU @ 1.10GHz  and 16 GB memory.

In our implementations, to ensure the consistency of the linear system, we set $b=Ax$, where $x$ is a vector with entries generated from a standard normal distribution. All computations are started from the initial vector $x^{(0)}=0$.  We stop the algorithms if the relative solution error (RSE)
$\frac{\|x\k-A^\dagger b\|^2_2}{\|A^\dagger b\|^2_2}\leq10^{-12}$. 
In order to compute $\gamma_1$ and $\gamma_2$, we first find the index $i_{\min}\in\arg\min\limits_{i\in[m]}\|a_i\|^2_2$. Then, we calculate $\gamma_1=\|A\|^2_F-\|a_{i_{\min}}\|^2_2$ and $\gamma_2=\gamma_1-\min\limits_{i\in[m],i\neq i_{\min}}\|a_i\|^2_2$.
	Furthermore, if it is possible to store $AA^\top$ at the initialization, Algorithm \ref{algo:main} could be faster in practice.  We refer the reader to \cite{bai2018greedy} for more details.

As is discussed in Remark \ref{xie-remark-3}, we mainly consider the linear systems with high coherence. Let
$$
\delta_{\max}:=\max_{ i,j\in [m],i\neq j}\left|\innL{\frac{a_i}{\nm{a_i}},\frac{a_j}{\nm{a_j}}}\right|
$$
and
$$
\delta_{\text{mean}}:=\frac{2}{m(m-1)}\sum_{1\leq i<j\leq m}\left|\innL{\frac{a_i}{\nm{a_i}},\frac{a_j}{\nm{a_j}}}\right|,
$$
where we define $\frac{0}{0}=0$.  It is obvious that $\delta_{\min}\leq\delta_{\text{mean}}\leq \delta_{\max}$, and $\delta_{\min}$ is defined as \eqref{def-delta}.
We use $\delta_{\min},\delta_{\text{mean}}$, and $\delta_{\max}$ to measure the coherence \cite{donoho2001uncertainty} between the rows of the coefficient matrices. Indeed, a larger value of $\delta_{\text{min}}$ suggests a higher coherence between any two rows of matrix $A$, while a larger value of $\delta_{\text{max}}$ indicates the presence of at least two rows in matrix $A$ that exhibit a high level of coherence. Moreover, if there are a large number of rows in matrix $A$ exhibiting high coherence, then the value of $\delta_{\text{mean}}$ is also large.
All the results are averaged over $20$ trials and we report the average number of iterations (denoted as Iter) and the average computing time in seconds (denoted as CPU). We use $\operatorname{Prob}\left(i_k=i\right)=\frac{\left|\tilde{r}\k_{i}\right|^2}{\left\|\tilde{r}\k\right\|_2^2}
$ as the  probability criterion for selecting the working row.

\subsection{Synthetic data}
In this test, we assign the entries of the coefficient matrix $A\in\mathbb{R}^{m\times n}$ as independent identically distributed uniform random variables within the interval $[t,1]$. It should be noted that varying the value of $t$ will appropriately change the coherence of $A$. As the value of $t$ approaches $1$, the correlation between the rows of $A$ becomes stronger, which means that the value of $\delta_{\min}$ becomes larger; See Figure \ref{figue1-1}.

\begin{figure}[hptb]
	\centering
	\begin{tabular}{cc}
		\includegraphics[width=0.45\linewidth]{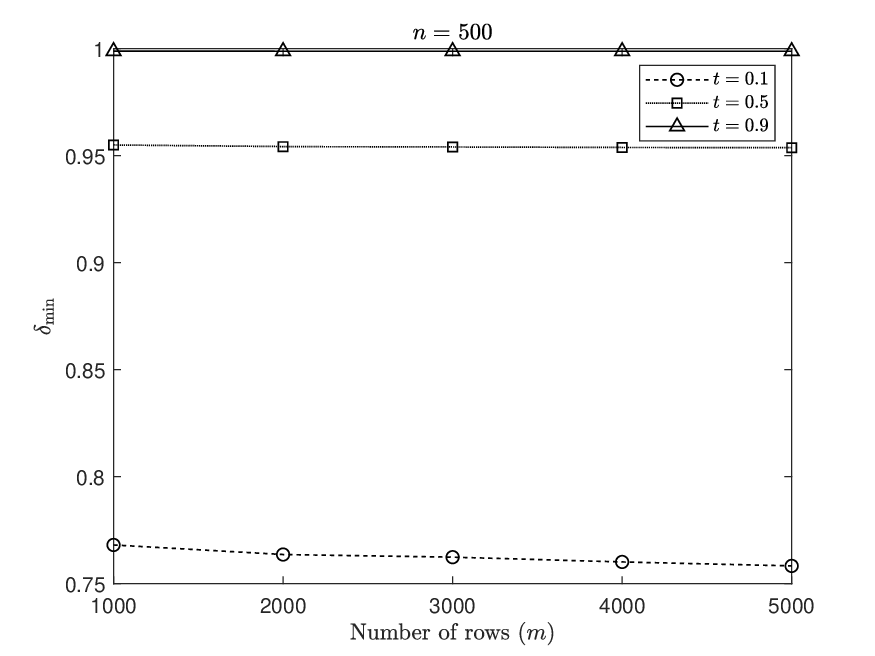}
		\includegraphics[width=0.45\linewidth]{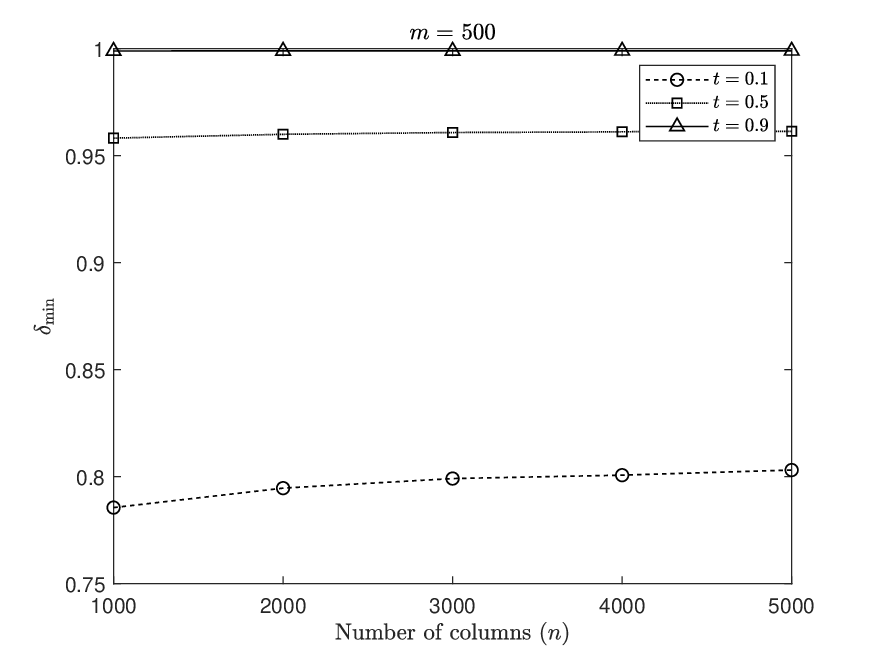}
	\end{tabular}
	\caption{Figures depict the value of $\delta_{\min}$ vs increasing number of rows ($n=500$) or columns ($m=500$). We set $t=0.1,0.5$, and $0.9$, respectively.}
	\label{figue1-1}
\end{figure}

Figures \ref{figue1} and \ref{figue2} illustrate the relationship between the CPU time, number of iterations, and the number of rows or columns of the matrix $A$ for GRK, MIRK, and GMIRK.
We have excluded the GRK method for $t=0.9$ due to its number of iterations exceeding $1,000,000$.
From these figures, it can be observed that the GMIRK method requires significantly fewer iteration steps and less computing time compared to both the GRK method and the MIRK method. The performance of the GRK method is highly sensitive to changes in $t$, whereas MIRK and GMIRK are not. As $t$ increases, the number of iteration steps and the CPU time of the GRK method increase rapidly, while those of the MIRK and the GMIRK method remain relatively stable or even decrease (see Figure \ref{figue1}).

\begin{figure}[hptb]
	\centering
	\begin{tabular}{cc}
		\includegraphics[width=0.45\linewidth]{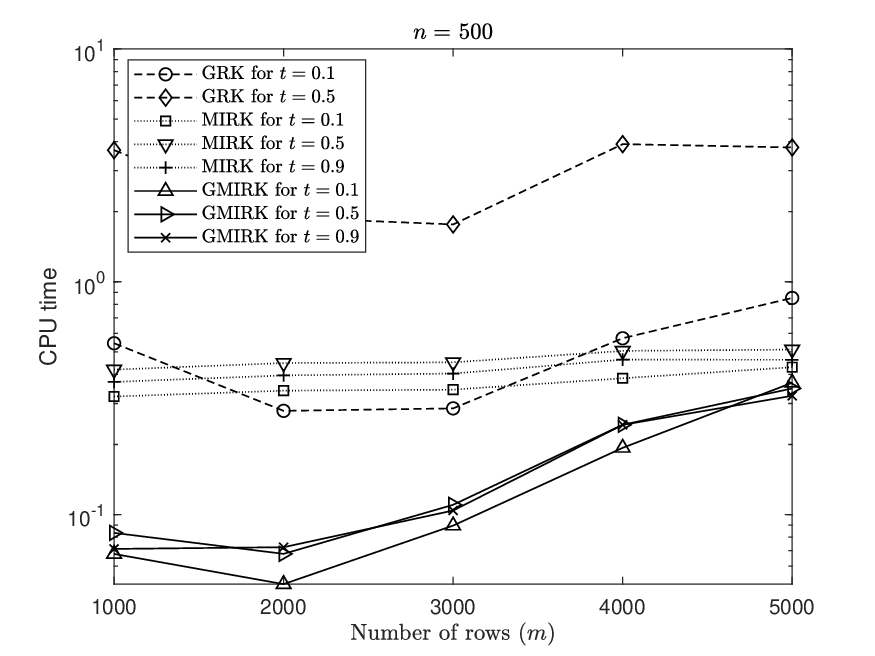}
		\includegraphics[width=0.45\linewidth]{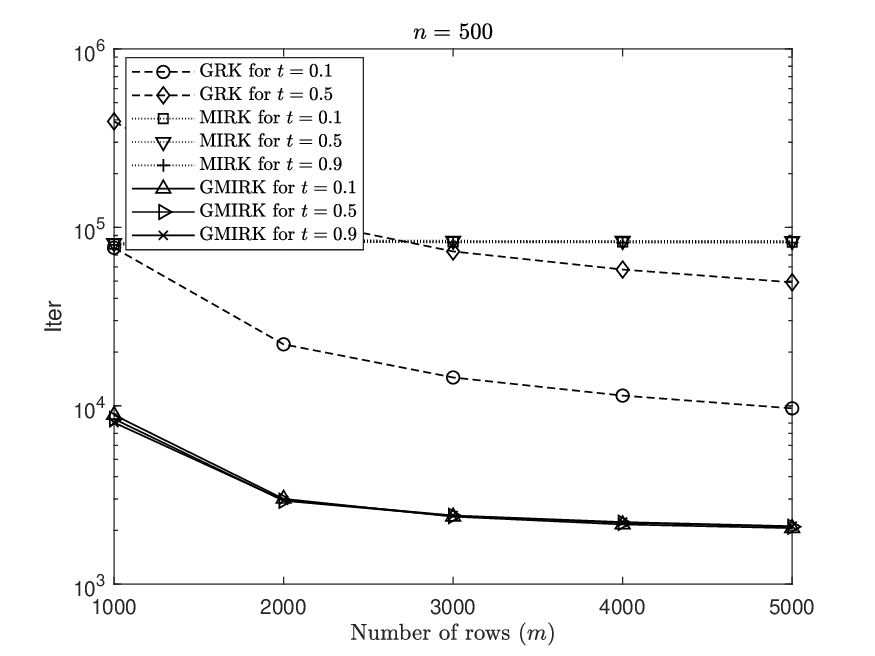}
	\end{tabular}
	\caption{Figures depict the CPU time (in seconds) and iteration vs increasing number of rows when $n=500$. We set $t=0.1,0.5$, and $0.9$, respectively. }
	\label{figue1}
\end{figure}

\begin{figure}[hptb]
	\centering
	\begin{tabular}{cc}
		\includegraphics[width=0.45\linewidth]{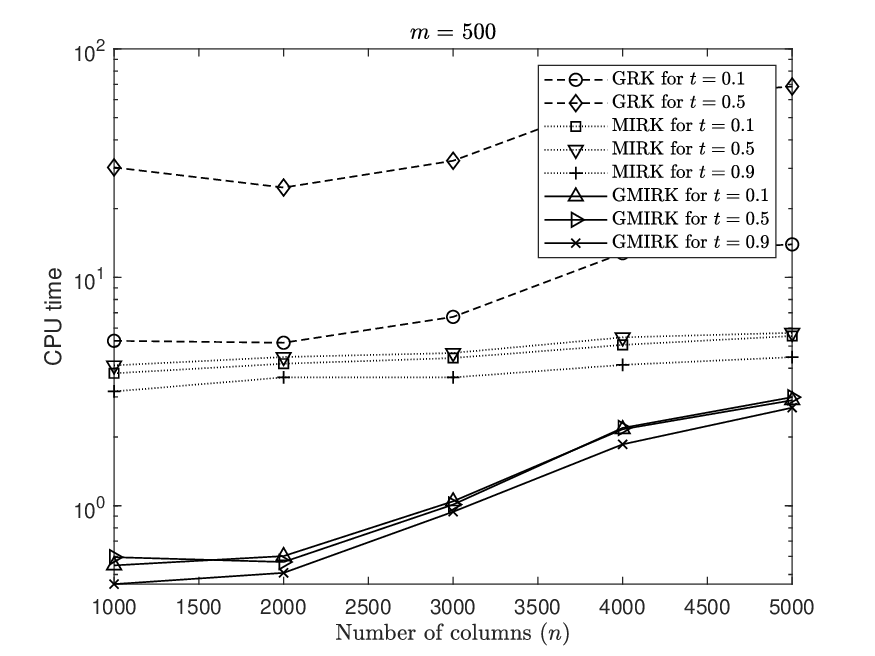}
		\includegraphics[width=0.45\linewidth]{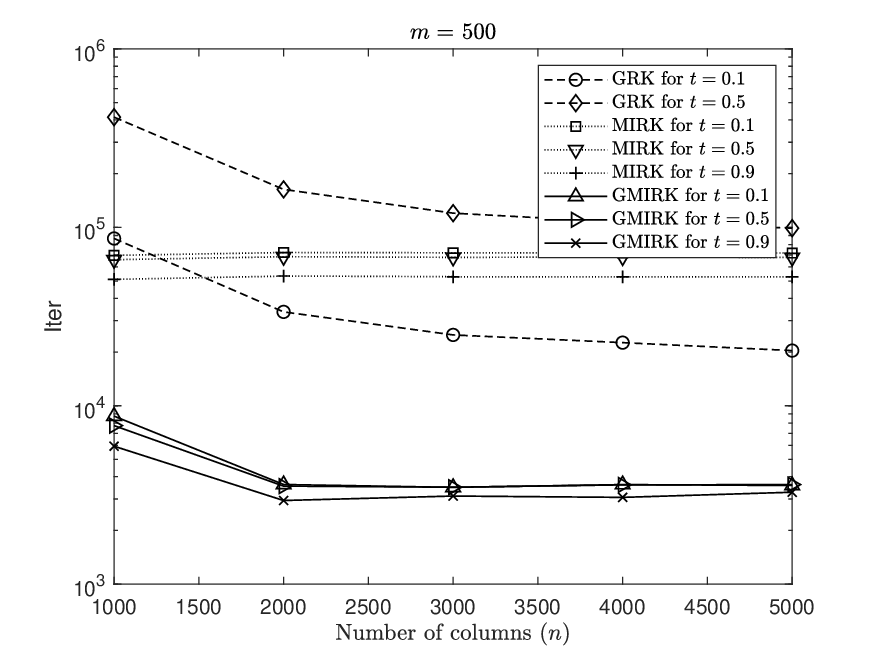}
	\end{tabular}
	\caption{ Figures depict the CPU time (in seconds) and iteration vs increasing number of columns when $m=100$.  We set $t=0.1,0.5$, and $0.9$, respectively. }
	\label{figue2}
\end{figure}

\subsection{Real-world data} The real-world data are available via the SuiteSparse Matrix Collection \cite{Kol19}. The matrices used in the experiments include {\tt bibd\_16\_8}, {\tt crew1}, {\tt WorldCities}, {\tt nemsafm}, {\tt model1}, {\tt ash958}, {\tt Franz1}, and {\tt mk10-b2}. Some of these matrices are full rank, while others are rank-deficient; See Table \ref{table1}.
Each dataset consists of a matrix $A\in\mathbb{R}^{m\times n}$ and a vector $b\in\mathbb{R}^m$.
In our experiments, we only use the matrices $A$ of the datasets and ignore the vector $b$.
In Table \ref{table2}, we report the number of iterations and the computing times for GRK, MIRK, and GMIRK. It is evident from the table that both GMIRK and GRK outperform the MIRK method. Furthermore, the GMIRK method demonstrates superiority over the GRK method in terms of the number of iterations in all test cases. Particularly, for the data sets with high coherence, such as {\tt bibd\_16\_8}, {\tt crew1}, and {\tt WorldCities}, the GMIRK method is at least two times faster than the GRK method in terms of CPU time.
However, for data sets with low coherence, such as {\tt model1}, {\tt ash958}, {\tt Franz1}, and {\tt mk10-b2}, the GRK method exhibits a slight advantage over the GMIRK method in terms of CPU time.
This is because both GMIRK and GRK now require almost the same number of iterations,  but GMIRK has a higher computational cost at each step compared to GRK.

\begin{table}
	\renewcommand\arraystretch{1.5}
	\setlength{\tabcolsep}{2pt}
	\caption{ The data sets from SuiteSparse Matrix Collection \cite{Kol19}. }
	\label{table1}
	\centering
	{\scriptsize
		\begin{tabular}{  |c| c| c| c| c |  }
			\hline
			Matrix & $m\times n$ &  rank &  $\frac{\sigma_{\max}(A)}{\sigma_{\min}(A)}$  & $(\delta_{\min},\delta_{\text{mean}},\delta_{\max})$    \\
			\hline
			{\tt bibd\_16\_8}& $120\times12870$ &  120  & 9.54 &    $(0.1648,0.2269,0.4286)$   \\
			\hline
			{\tt crew1} & $135\times6469 $ &  135  &18.20  &   $(0,0.0463,0.9329)$  \\
			\hline
			{\tt WorldCities} & $315\times100$ &  100  &6.60  &   $(0,0.3333,1)$  \\
			\hline
			{\tt nemsafm} & $334\times 2348$ &   334  &4.77  &   $(0,0.0018,0.4082)$   \\
			\hline
			{\tt model1} & $  362\times798 $ &  362  & 17.57 &$(0,0.0028,0.7069)$   \\
			\hline
			{\tt ash958} & $958\times292$ &  292  &3.20  &   $(0,0.0063,0.5000)$   \\
			\hline
			{\tt Franz1} & $ 2240\times768 $ &  755  & 2.74e+15 &  $(0,0.0025,0.5000)$  \\
			\hline
			{\tt mk10-b2} & $  3150\times630 $ &  586 & 2.74e+15 &$(0,0.0044,0.3333)$   \\
			\hline
		\end{tabular}
	}
\end{table}

\begin{table}
	\renewcommand\arraystretch{1.5}
	\setlength{\tabcolsep}{2pt}
	\caption{ The average Iter and CPU of GRK, MIRK, and GMIRK for solving linear systems with coefficient matrices in Table \ref{table1}. }
	\label{table2}
	\centering
	{\scriptsize
		\begin{tabular}{  |c| c|  c| c |c |c |c|  }
			\hline
			\multirow{2}{*}{ Matrix}  &\multicolumn{2}{c| }{GRK} &\multicolumn{2}{c| }{MIRK}  &\multicolumn{2}{c| }{GMIRK}
			\\
			\cline{2-7}
			&   Iter&   CPU & Iter & CPU    & Iter & CPU      \\
			\hline
			{\tt bibd\_16\_8}& 2168.90  &   4.0391 &  5941.70  &   4.5654 &  1226.80  &   {\bf 2.0495}    \\
			\hline
			{\tt crew1} & 6100.00  &   2.6995 & 28178.40  &   9.3005 &  2475.40  &   {\bf0.9595}   \\
			\hline
			{\tt WorldCities} & 15063.80  &   1.2245 & 58830.90  &   2.2064 &  6653.50  &  {\bf 0.5965}   \\
			\hline
			{\tt nemsafm} & 2604.10  &   0.2382 & 31668.30  &   2.4590 &  2423.40  &   {\bf0.2331}  \\
			\hline
			{\tt model1} & 11583.80  &  {\bf 0.6487} & 117828.80  &   5.8190 & 10194.50  &   0.6841  \\
			\hline
			{\tt ash958} & 1615.00  &   {\bf0.0768} & 12371.50  &   0.4878 &  1562.70  &   0.0910  \\
			\hline
			{\tt Franz1} & 15654.20  &   {\bf1.1031} & 71668.60  &   4.6155 & 14635.50  &   1.1963  \\
			\hline
			{\tt mk10-b2} & 2338.90  & {\bf  0.1910} & 16803.40  &   1.2553 &  2342.10  &   0.2173   \\
			\hline
		\end{tabular}
	}
\end{table}

	\section{Concluding remarks}

This paper presented the GMIRK method, an accelerated variant of the randomized Kaczmarz method, for solving large-scale linear systems. The method combines two techniques: the  greedy probability criterion and the multi-step inertial extrapolation approach. Moreover, we introduced a tighter threshold parameter for the greedy probability criterion.
We showed that the GMIRK method converges faster than the GRK method in both theory and experiments.
Furthermore, we elaborated the geometric interpretation of the multi-step inertial extrapolation approach, establishing its connection with skectch-and-project method \cite{15M1025487,19M1285846} as well as oblique projection technique \cite{LI2022100342}.

In \cite{Lev10}, Leventhal and Lewis  have extended the randomized Kaczmarz method to solve the linear feasibility problem $Ax\leq b$. It would be interesting to explore the potential benefits of combining inertial extrapolation with the greedy probability criterion for solving the linear feasibility problem.

\bibliographystyle{plain}
\bibliography{bibliography}

\end{document}